\documentclass[11pt]{amsart}  
\usepackage[english]{babel}
\usepackage[utf8]{inputenc}
\usepackage[T1]{fontenc}

\usepackage[top=3cm,bottom=2.5cm,left=3cm,right=3cm,marginparwidth=1.75cm]{geometry}

\usepackage[usenames,dvipsnames,svgnames,table]{xcolor}
\usepackage{mathtools,fullpage,hyperref,caption,enumitem }
\usepackage{amsfonts,amsthm,amsmath,amssymb}
\usepackage{tikz}
\usepackage{xargs}

\newtheorem{theorem}{Theorem}[section]

\newtheorem{proposition}[theorem]{Proposition}
\newtheorem{remark}[theorem]{Remark}

\theoremstyle{definition}
\newtheorem{definition}[theorem]{Definition}
\newtheorem{example}[theorem]{Example}

\newtheorem{corollary}[theorem]{Corollary}

\newcommand{\len}[3]{\ell_{#1}\left(#2;#3\right)}
\newcommand{\ileft}[3]{\texttt{left}_{#1}\left(#2;#3\right)}
\newcommand{\iright}[3]{\texttt{right}_{#1}\left(#2;#3\right)}

\title{Counting $k$-Naples parking functions through\\ permutations and the $k$-Naples area statistic}

\author{Laura Colmenarejo}
\address[L.\ Colmenarejo]{Department of Mathematics and Statistics, Umass Amherst, United States}
\email{\textcolor{blue}{\href{mailto:laura.colmenarejo.hernando@gmail.com}{laura.colmenarejo.hernando@gmail.com}}}

\author{Pamela E.\ Harris}
\address[P.\ E.\ Harris]{Department of Mathematics and Statistics, Williams College, United States}
\email{\textcolor{blue}{\href{mailto:peh2@williams.edu}{peh2@williams.edu}}}

\author{Zakiya Jones}
\address[Z.\ Jones]{Department of Mathematics, Pomona College, United States}
\email{\textcolor{blue}{\href{mailto:zakiyacmjones@gmail.com}{zakiyacmjones@gmail.com}}}

\author{Christo Keller}
\address[C.\ Keller]{Department of Mathematics and Statistics, UMass Amherst, United States}
\email{\textcolor{blue}{\href{mailto:thechristokellera@gmail.com}{thechristokellera@gmail.com}}}

\author{Andr\'es Ramos Rodr\'{i}guez}
\address[A.\ Ramos Rodr\'{i}guez]{Department of Mathematics, Universidad de Puerto Rico, R\'{i}o Piedras, United States}
\email{\textcolor{blue}{\href{mailto:ramosandres443@gmail.com}{ramosandres443@gmail.com}}}

\author{Eunice Sukarto}
\address[E.\ Sukarto]{Department of Mathematics, University of California, Berkeley, United States}
\email{\textcolor{blue}{\href{mailto:eunicesukarto@berkeley.edu}{eunicesukarto@berkeley.edu}}}

\author{Andr\'es R.\ Vindas-Mel\'endez}
\address[A.\ R.\ Vindas-Mel\'endez]{Department of Mathematics, University of Kentucky, United States}
\email{\textcolor{blue}{\href{mailto:andres.vindas@uky.edu}{andres.vindas@uky.edu}}}
\begin{document}

\begin{abstract}
We recall that the $k$-Naples parking functions of length $n$ (a generalization of parking functions) are defined by requiring that a car which finds its preferred spot occupied must first back up a spot at a time (up to $k$ spots) before proceeding forward down the street. Note that the parking functions are the specialization of $k$ to $0$.
For a fixed $0\leq k\leq n-1$, we define a function $\varphi_k$
which maps a $k$-Naples parking function to the permutation denoting the order in which its cars park.
By enumerating the sizes of the fibers of the map $\varphi_k$
we give a new formula for the number of $k$-Naples parking functions as a sum over the permutations of length $n$.

We remark that our formula for enumerating $k$-Naples parking functions is not recursive, in contrast to the previously known formula of Christensen et al~\cite{NaplesPF}.
It can be expressed as the product of the lengths of particular subsequences of permutations, and its 
specialization to $k=0$ gives a new way to describe the number of parking functions of length $n$.
We give a formula for the sizes of the fibers of the map $\varphi_0$, and we provide a recurrence relation for its corresponding logarithmic generating function.
Furthermore, we relate the $q$-analog of our formula 
to a new statistic that we denote $\texttt{area}_k$ and call the
\textit{$k$-Naples area statistic}, the specialization of which to $k=0$ gives the \texttt{area} statistic on parking functions.
\end{abstract}

\maketitle

\section{Introduction}
Parking functions are combinatorial objects introduced in~\cite{KW} to study hashing problems. 
To define these objects, consider a one-way street with $n$ parking spots labeled $1$ though $n$ and a line of $n$ cars, $c_1, c_2, \dots, c_n$ waiting to park. 
A parking preference is a tuple $\alpha = (a_1, \dots, a_n)\in[n]^n$, where $a_i$ is the preferred parking spot of car $c_i$. 
The rule for parking cars is that car $c_i$ goes to its preferred spot $a_i$ and if it is empty, parks there. 
If the space is occupied, $c_i$ moves forward until it finds the next available space. 
We say that a parking preference $\alpha$ is a parking function if all the cars can park within the given $n$ spots without any cars driving off the road.

Parking functions have received much attention in the combinatorics literature~\cite{bonaM}.
The study of parking functions has led to connections with fields such as graph theory, representation theory, hyperplane arrangements, and discrete geometry~\cite{graphs,dh,hyperplanes, PitmanStanley}.
Moreover, in the last few years many generalizations of parking functions have appeared in the literature \cite{Adeniran1+,Adeniran4+,Adeniran2+,Adeniran3+}.  For a survey of generalizations and open problems related to parking functions, see in~\cite{PFCYOA}. 
Among those generalizations, we find the \textbf{$k$-Naples parking functions}, introduced by Baumgardner for $k=1$ and by Christensen et.\ al.\  for general $k$, see~\cite{BG,NaplesPF}. This generalization modifies the parking rule so that car $c_i$, upon finding its preferred space $a_i$ occupied, first backs up and checks if the spot $a_i - 1$ is occupied. If it is empty it parks there, otherwise it backs up, a spot at a time, at most $k$ spaces attempting to park in the first available before going forward once again. The set of $k$-Naples parking functions is the set of parking functions that can park under the $k$-Naples rule. In~\cite{NaplesPF}, the authors study this generalization and provide a recursive formula for the number of $k$-Naples parking functions and provide some connections to signature Dyck paths.

This paper continues the study of the $k$-Naples parking functions by extending it to a new $q$-analog. To make concrete our results we begin by describing our process.
Let $S_n$ denote the symmetric group on $n$ letters, $PF_{n,k}$ denote the set of $k$-Naples parking functions of length $n$, and $PF_{n,0}:=PF_n$ denote the set of parking functions of length $n$. 
For a fixed $0\leq k\leq n-1$, we define a map $\varphi_k$ which maps a $k$-Naples parking function to the permutation denoting the order in which the cars park (see Definition~\ref{def:varphi_k}). 
Then we count the size of the fibers of this map under each permutation in $S_n$ (see Theorem~\ref{thm:permutations-k}), thereby giving a formula for the number of $k$-Naples parking functions as a sum over $S_n$ (see Theorem~\ref{thm:pfnk}). 
This formula enumerating $k$-Naples parking functions is new and not recursive, in contrast to the one in~\cite{NaplesPF}. Moreover, our formula is given as the product of the lengths of particular subsequences of permutations. 
In the parking function case, that is when $k=0$, we also give a recurrence relation for the fibers of the map $\varphi_0$, providing a recursive formula to compute the coefficients of the corresponding generating function.

From the above mentioned results, we continue our work by studying a $q$-statistic of $k$-Naples parking functions of length $n$. 
Given a non-negative integer $n$, we denote its $q$-analog 
$$ [n]_q =\lim_{q\rightarrow 1} \dfrac{1-q^n}{1-q} =  1+q+q^2+\cdots +q^{n-1}.$$ 
With this $q$-analog, one may define $q$-analogs of almost any formula. 
For instance, the $q$-factorial or the $q$-binomial coefficients.
In general, we denote by $[\star]_q$ the $q$-analog of the formula $\star$ by substituting any number appearing in $\star$ by its $q$-analog.
It turns out that the $q$-analog of our formula for counting parking functions coincides with the \textit{$q$-fermionic formula} presented in~\cite{dh} (taking $t=0$). That is, we describe the relation between the distribution of the \texttt{area} statistic and the $q$-analog of the lengths involved in our formula (see Proposition~\ref{prop:classical-q}). We also give the definition of the \textbf{$k$-Naples area statistic}, as well as the analogous result for our formula for $k$-Naples parking functions.

The paper is organized as follows. In Section~\ref{section:background}, we include main definitions and results related to parking functions and $k$-Naples parking functions. In Section~\ref{section:permutations}, we present our main result counting $k$-Naples parking functions through permutations, and we include the result for parking functions to illustrate the intuitive idea of the proof. 
As a continuation, in Section~\ref{section:GFvarphi} we give a recursive formula to enumerate permutations whose preimage under $\varphi_0$ has a given cardinality and from this we present the associated generating function.
Finally, in Section~\ref{section:qanalogs}, we define the $\texttt{area}_k$ statistic for $k$-Naples parking functions and we relate our enumerative results to the $\texttt{area}$ statistic.

\section{Background}\label{section:background}
Let us start by defining the main objects of our study, as well as the notation we use. 
\begin{definition}
Let $\alpha= (a_1,\dots,a_n)$, with $a_i\in[n]:=\{1,\dots,n\}$. 
We say that $\alpha$ is a \textbf{parking preference sequence} meaning that the $i^{\text{th}}$ car, henceforth denoted $c_i$, wants to park in spot $a_i$. 
Then, $\alpha$ is a \textbf{parking function of length $n$} if all the cars park under the following \textbf{parking rule}:
\begin{quote}
Imagine $n$ cars travel down a one-way street with $n$ parking spots.  Each car prefers a spot, which it attempts to park in.  If the spot is empty, it parks there and succeeds; otherwise, it continues down the road until it finds an empty spot.
\end{quote}\end{definition}

We denote by $PP_n$ the set of parking preferences of length $n$, and by $PF_n$ the set of parking functions of length $n$. To illustrate the above definition we present the following example. 

\begin{example}
Consider the parking preference $(2,1,1)$.
We have three open spots and $c_1$ takes the $2^{\text{nd}}$ spot.  
In comes $c_2$ and parks in spot $1$.  
Now, $c_3$ drives up to the $1^\text{st}$ spot and sees that the parking spot is occupied. 
Following the parking rule, it proceeds to the $2^\text{nd}$ spot, which is also taken. 
Finally, it tries the $3^\text{rd}$ spot, where it parks successfully. 
All three cars park, so $(2,1,1)\in PF_{3}$.
\end{example}

The $k$-Naples parking functions generalize these objects by modifying the parking rule. 

\begin{definition}
We say that a parking preference $\alpha$ is a \textbf{$k$-Naples parking function}, with $0\leq k \leq n$, if all the cars are able to park under the following \textbf{$k$-Naples parking rule}:
\begin{quote}
Imagine $n$ cars travel down a one-way street with $n$ parking spots. Each car prefers a spot, which it attempts to park in.  
If the spot is empty, it backs up checking up to $k$ spots behind its preferred spot and parks in the first available. 
If all the $k$ spots preceding its preferred spot are occupied, then the car continues down the street until it finds an empty spot in which to park.
\end{quote}
We denote by $PF_{n,k}$ the set of $k$-Naples parking functions of length $n$.
Moreover, we denote by $PF_{n,k}^d$ the subset of $PF_{n,k}$ given by the decreasing $k$-Naples parking functions of length $n$, i.e., the set of all $\alpha = (a_1,\ldots, a_n) \in PF_{n,k}$ such that $a_1\geq a_2\geq \cdots \geq a_n$. 
\end{definition}

Note that the case when $k=0$ is precisely the case where cars do not back up, which is exactly the definition of a parking function. Hence  $PF_{n,0}=PF_n$. 
Also, if the cars are able to back up and check up to $n-1$ spots behind their preferred spot, this would allow cars to check the entire length of the street in search for an empty parking spot. Hence, every car can park regardless of their parking preference, which implies that $PF_{n,n-1}=PP_n$. Moreover, whenever $1 \leq k\leq n$ it follows that $PF_{n,k-1}\subseteq PF_{n,k}$. But not every $k$-Naples parking functions is a parking function.

\begin{example}
Take the parking preference $(3,2,2)$.
Under the parking rule: the first two cars park, but $c_3$ drives off the road. 
Following the $1$-Naples parking rule, the first two cars park (as before) and $c_3$ finding the $2^\text{nd}$ parking spot occupied, it subsequently checks the $(2-1)^\text{th}=1^{\text{st}}$ spot, where it is able to park.  
Therefore $(3,2,2)\in PF_{3,1}$ even though $(3,2,2)\not\in PF_{3,0}=PF_3$.
\end{example}

Parking functions are in bijection with many other combinatorial objects. 
For instance, rooted forests~\cite{rforests}, maximal chains of non-crossing partitions~\cite{ncp}, and regions in the Shi arrangement~\cite{hyperplanes}.
For our purposes, we present their bijection to labeled Dyck paths.

\begin{definition}
A \textbf{Dyck path of length $2n$} is a lattice path from $(0, n)$ to $(n, 0)$ consisting of $n$ steps east by $(1,0)$ and $n$ steps south by $(0,-1)$ all of which have the path staying above the diagonal $y =n- x$. 
A \textbf{labeled Dyck path of length $2n$} is a Dyck path such that the south steps are labeled with numbers $1,\dots, n$ and consecutive south steps have increasing labels. 
\end{definition}

\begin{theorem}[Stanley] \label{thm:dyckPFbijection}
The set of parking functions of length $n$ are in bijection with the set of labeled Dyck paths of length $2n$. 
\end{theorem}

In the proof of Theorem~\ref{thm:dyckPFbijection}, given a parking function $\alpha = (a_1,a_2,\dots, a_n)\in PF_n$, let $u_i$ denote the number of occurrences of $i$ in $\alpha$.
Consider an $n\times n$ grid, and label the vertical lines from left to right with $1,\dots, n$.
Notice the last vertical line on the right is not labeled. 
Then, we define the Dyck path $P$ by drawing $u_i$ south steps in the $i^{\text{th}}$ column and we label each south step with the position of the value $i$ in $\alpha$, making sure that the labels are increasing for consecutive south steps. 

Next, we illustrate the bijection of Theorem~\ref{thm:dyckPFbijection} with the following example. 
\begin{example}
Consider the parking function $\alpha= (3,3,1,4,2,2)$. Then, $u_1$ with label $\{3\}$, $u_2=2$ with labels $\{5,6\}$, $u_3=2$ with labels $\{1,2\}$, $u_4=1$ with label $\{4\}$, and $u_5=u_6=0$. In Figure~\ref{fig:dyckpath}, we draw the labeled Dyck path corresponding to $\alpha$.

\begin{figure}[h]
    \centering
    \begin{tikzpicture}[scale=0.8]
(0,0) rectangle +(6,6);
\draw[help lines] (0,0) grid +(6,6);
\draw[dashed] (0,6) -- (6,0);
\coordinate (prev) at (0,0);
\draw [color=black, line width=2] (0,6)--(0,5)--(1,5)--(1,3)--(2,3)--(2,1)--(3,1)--(3,0)--(6,0);
\draw (-0.25,5.5) node {3};
\draw (0.75,4.5) node {5};
\draw (0.75,3.5) node {6};
\draw (1.75,2.5) node {1};
\draw (1.75,1.5) node {2};
\draw (2.75,0.5) node {4};
 \end{tikzpicture}
    \caption{Labeled Dyck path corresponding to the parking function $\alpha=(3,3,1,4,2,2)$.}
    \label{fig:dyckpath}
\end{figure}
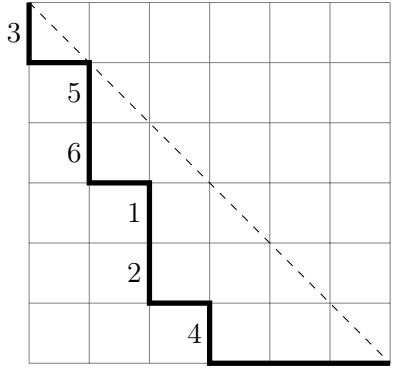
\end{example}

There is no known generalization of this result to all the $k$-Naples parking functions. 
However, using what we call $k$-lattice paths of length $2n$, we have the following partial characterization for just the weakly decreasing $k$-Naples parking functions of length $2n$.

\begin{definition}
Given $n,k\in \mathbb{N}$ with $0 \leq   k \leq  n - 1$, a \textbf{$k$-lattice path of length $2n$} is a lattice path from $(0, n)$ to $(n, 0)$ consisting of $n$ steps east by $(1, 0)$ and $n$ steps south by $(0, -1)$, such that the path never goes above the line $y = n - x + k$.
\end{definition}

\begin{theorem}[Theorem 1.3,~\cite{NaplesPF}]\label{thm:bijectionNPF}
Given $n,k\in \mathbb{N}$ with $k \leq  n - 1$, the set of decreasing $k$-Naples parking functions of length $n$ is in bijection with the $k$-lattice paths of length $2n$ where the first step is south.
\end{theorem}

\begin{remark}
The definition of $k$-lattice paths does not require that the \emph{first step is south}. 
However, it is a necessary condition for the bijection, as the authors of~\cite{NaplesPF} point out in the proof.
\end{remark}

In the most general case, parking preferences are in bijection with lattice paths.
This bijection sends a parking preference, $\alpha =  (a_1,\ldots, a_n) \in PP_n $, to the lattice path with east steps $(i-1, a_i-1)$ to $(i,a_i-1)$, for $i=1,\dots, n$. 
In~\cite[Theorem 1.3]{NaplesPF}, the authors show that this bijection restricts nicely to $k$-Naples parking functions establishing Theorem~\ref{thm:bijectionNPF}. We consider the following example to illustrate this bijection.

\begin{example}
In Figure~\ref{fig:pathexample}, we draw the $2$-lattice path corresponding to the $2$-Naples parking function $\alpha = (6,6,4,4,2,2)$.
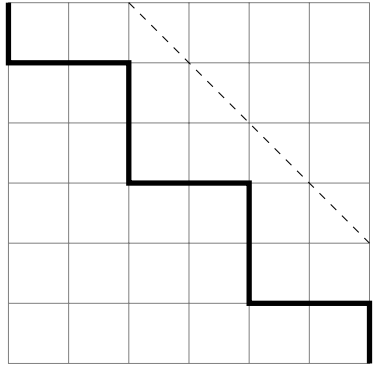
\begin{figure}[ht]
    \centering
    \begin{tikzpicture}[scale=0.8]
(0,0) rectangle +(6,6);
\draw[help lines] (0,0) grid +(6,6);
\draw[dashed] (2,6) -- (6,2);
\coordinate (prev) at (0,0);
\draw [color=black, line width=2] (0,6)--(0,5)--(2,5)--(2,3)--(4,3)--(4,1)--(6,1)--(6,0);
 \end{tikzpicture}
    \caption{2-lattice path corresponding to $\alpha=(6,6,4,4,2,2)$.}
    \label{fig:pathexample}
\end{figure}
\end{example}

Note that $a_i$ corresponds to the $i^{\text{th}}$ east step. Moreover, the fact that the first step is south implies that for all $i$, $a_i \leq \min(n, n - i + k + 1)$. This gives us the following variant of Theorem~\ref{thm:bijectionNPF}:
\begin{theorem}\label{theorem:version2}
Let $\alpha = (a_1, \dots, a_n)\in PP_n$ be a decreasing parking preference. Then, $\alpha \in PF_{n, k}^d$ if and only if $a_i \leq \min(n, n -i + 1 + k)$ for all $1\leq i\leq n$. 
\end{theorem}

Formulas enumerating the number of parking functions and $k$-Naples parking functions are presented below.
\begin{theorem}[\cite{KW,Pyke}]
For $n\geq 1$,
$$|PF_n| = (n+1)^{n-1}.$$
\end{theorem}
In the next section we give a new formula to enumerate the number of $k$-Naples parking function. We remark that prior to the present work, the only other result enumerating $k$-Naples parking function consists of the following recursive formula.
\begin{theorem}[Theorem 1.1,~\cite{NaplesPF}]\label{thm:RecEnumNPF}
For $n\geq 0$ and $0\leq k \leq n$,
$$|PF_{n+1,k}| = \sum_{i=0}^n \binom{n}{i} \min(i+1+k, n+1) |PF_{i,k}| (n-i+1)^{n-i-1}.$$
\end{theorem}

\section{Counting  \texorpdfstring{$k$}{k}-Naples parking functions through permutations}\label{section:permutations} 

In this section, we provide a non-recursive formula for $|PF_{n,k}|$ as a sum over permutations.
We begin by establishing the formula for the case $k=0$, the case of parking functions, as it helps build intuition to generalize the result to $k$-Naples parking functions for $k>0$.

\subsection{The parking functions case}\label{subsect:PFandpermutations}
Given a parking function, we consider the permutation, written in \textit{one-line notation}, resulting from recording the position in which each car parks. 

\begin{definition}\label{def:varphi}
Given a parking function $\alpha$, consider the function $\varphi: PF_n \to S_n$ which maps $\alpha = (a_1,a_2,\ldots, a_n) \in PF_n$ to $\varphi(\alpha) = s_1s_2\cdots s_n$, where parking spot $i$ is occupied by the ${s_i}^{\text{th}}$ car.
\end{definition}
We remark that the map $\varphi$ is the inverse of the \textit{outcome map} defined in~\cite{Colaric2020IntervalPF} and slightly different from the definition appearing in \textit{FindStat St001346}~\cite{FindStat}.
However, our results follow more naturally by defining $\varphi$ as in Definition~\ref{def:varphi}.

In order for a parking preference to be a parking function, all the cars have to park and there cannot be two cars in the same parking space, therefore $\varphi$ is well-defined. 
Moreover, every permutation is a parking function.
In fact, $\varphi$ maps each permutation to itself, so $\varphi$ is surjective.  To illustrate Definition \ref{def:varphi} we consider the following example.
\begin{example}\label{example:permutations1}
Let $\alpha = (4,2,2,4,1)\in PF_5$, so the cars $c_1,c_2,c_3,c_4,c_5$ park in the following way 
$$\begin{array}{c|c|c|c|c}
\text{spot } 1 & \text{spot } 2 & \text{spot } 3 & \text{spot } 4 & \text{spot } 5\\\hline
c_5 & c_2 & c_3 & c_1 & c_4 
\end{array} 
$$
Therefore, $\varphi(\alpha) = 52314$.
\end{example}

Now that we have a well-defined map from parking functions to permutations, a natural question to ask is how many parking functions map to a given permutation. 
In other words, given $\sigma \in S_n$, we would like to determine the size of the fiber $\left|\varphi^{-1}(\sigma)\right|$. 
The following result answers the aforementioned question.
\begin{proposition}\label{prop:permutation}
Let $\sigma = s_1\cdots s_n\in S_n$ be a permutation. 
Then $$|\varphi^{-1}(\sigma)|=\prod_{i=1}^n \len{}{i}{\sigma}$$
where $\len{}{i}{\sigma}$ is the length of the longest subsequence $s_j\cdots s_i$ of $\sigma$ such that $s_t\leq s_i$ for all $j\leq t\leq i$.  
\end{proposition}

Before establishing Proposition~\ref{prop:permutation}, we illustrate the idea of the proof with an example.
\begin{example}\label{ex:example3.4}
Let $\sigma = 23514$.
We want to count the number of parking functions $\alpha$ such that $\varphi(\alpha)=\sigma$ by determining all possible entries of $\alpha$ given the entries of $\sigma$. To begin, we let entries of $\sigma=s_1s_2s_3s_4s_5$ and consider each $s_i$ for $1\leq i\leq 5$ individually:
\begin{itemize}
\item[-] Since $s_1=2$, we know that $c_2$ parked in spot~1. This implies that $c_2$ preferred the parking spot $1$. Thus $\len{}{1}{23514}=1$, which is the length of the subsequence $2$ in~$\sigma$. 

\item[-] Since $s_2=3$, we know that $c_3$ parked in spot 2. 
This implies that $c_3$ preferred the parking spots $1$ or $2$. 
Thus $\len{}{2}{23514}= 2$, which is the length of the subsequence $23$ in~$\sigma$. 

\item[-] Since $s_3=5$, we know that $c_5$ parked in spot 3. 
This implies that $c_5$ preferred the parking spots $1$, $2$, or $3$. 
Thus $\len{}{3}{23514}= 3$, which is the length of the subsequence $235$ in~$\sigma$. 

\item[-] Since $s_4=1$, we know that $c_1$ parked in spot 4. 
This implies that $c_1$ preferred the parking spot $4$. 
Thus $\len{}{4}{23514}= 1$, which is the length of the subsequence $1$ in~$\sigma$. 

\item[-] Finally, since $s_5=4$, we know that $c_4$ parked in spot 5. 
This implies that $c_4$ preferred the spot $4$ or $5$. 
Thus $\len{}{5}{23514} = 2$, which is the length of the subsequence $14$ in $\sigma$. 
\end{itemize}

Multiplying all values of $\len{}{i}{23514}$ for $1\leq i\leq 5$, we have 12 distinct parking functions of length 5 mapping to $\sigma=23514$ via the map $\varphi$. We list these 12 parking functions below:
$$
\begin{array}{cccccc}
41141, & 41142, & 41143, & 41151, & 41152, & 41153, \\
41241, & 41242, & 41243, & 41251, & 41252, & 41253. 
\end{array}
$$
\end{example}
We are now ready to establish Proposition~\ref{prop:permutation}. 
\begin{proof}[Proof of Proposition~\ref{prop:permutation}]\label{proof:PermProp}
Let $\alpha = (a_1,\ldots, a_n)\in \varphi^{-1}(\sigma)$, $\sigma=s_1s_2\cdots s_n\in S_n$, and define $\pi \in S_n$ in such a way that $s_{\pi(i)} = i$ for all $i\in[n]$. That is, $c_i$ parks in the parking spot $\pi(i)$. Note that this corresponds to the \emph{outcome map}~\cite{Colaric2020IntervalPF}.
Now, since $\pi$ is a permutation in $S_n$, we have that
\[
\prod_{i=1}^n \len{}{\pi(i)}{\sigma} = \prod_{i=1}^n \len{}{i}{\sigma}
\]
as the product of the lengths only changes the order of the terms.
Thus, it suffices to show
\begin{align}\label{eq:ellpi}
|\varphi^{-1}(\sigma)| &= \prod_{i=1}^n \len{}{\pi(i)}{\sigma}.
\end{align}
To do so, for $i\in[n]$ we claim that $\len{}{\pi(i)}{\sigma}$ is the number of possible preferred parking spots $a_i$ of $c_i$ satisfying $\varphi(\alpha) = \sigma$. Proving this then implies that \eqref{eq:ellpi} follows by taking the product over all possibilities for $a_1,\ldots, a_n$.

Let us prove our claim:
\begin{center}
\textit{Given $c_i$, $\len{}{\pi(i)}{\sigma}$ is the number of possible preferred parking spots $a_i$ such that $\varphi(\alpha) = \sigma$}.
\end{center}
To establish this claim we prove the following:
\begin{enumerate}
    \item[Fact 1:] Every possible $a_i$ is an index of the longest subsequence $s_j\cdots s_{\pi(i)}$ of $\sigma$ such that $s_t\leq s_{\pi(i)}$ for all $j\leq t \leq \pi(i)$.
    \item[Fact 2:] Every index of the longest subsequence $s_j\cdots s_{\pi(i)}$  of $\sigma$ such that $s_t\leq s_{\pi(i)}$ for all $j\leq t \leq \pi(i)$ gives a possible $a_i$.
\end{enumerate} 
We begin by establishing Fact 1.
By definition, $c_i$ parks in the spot numbered $\pi(i)$. 
Otherwise, $c_i$ would not be able to park in spot $\pi(i)$. More explicitly, we must have either
\begin{enumerate}
    \item[(i)] $c_i$ parks in its preferred spot, and so $a_i = \pi(i)$, or
    \item[(ii)] $c_i$ tries to park in a previous spot $a_i < \pi(i)$ but the parking spots $a_i,\ldots, \pi(i)-1$ are all occupied. 
    Thus $\{a_i,\ldots, \pi(i)-1\} \subset \{\pi(1), \ldots, \pi(i-1)\}$, so $\{s_{a_i},\ldots, s_{\pi(i)-1}\} \subset \{s_{\pi(1)}, \ldots, s_{\pi(i-1)}\}\\=\{1,\ldots,i-1\}$.
    This means that for any choice of $t$ satisfying $a_i\leq t<\pi(i)$, there exists $r\in \{1,\ldots, i-1\}$ such that $s_t = s_{\pi(r)} = r < i = s_{\pi(i)}$. 
    
    \end{enumerate} 

Conditions (i) and (ii) imply that for all $t$ satisfying $a_i\leq t\leq \pi(i)$, we have $s_t \leq s_{\pi(i)}$, where the equality arises from condition (i). 
Hence, we have a necessary condition for $a_i$, which completes the proof of Fact 1. 

Now we prove Fact 2. 
Consider $s_j\cdots s_{\pi(i)}$ such that $s_t\leq s_{\pi(i)}$, for all $j\leq t\leq \pi(i)$. 
This means spots $j,\ldots, \pi(i)-1$ have been occupied before $c_i$ attempts to park.
If $a_i=\pi(i)$, then $c_i$ parks in spot $\pi(i)$ and we are done. 
If $j\leq a_i< \pi(i)$, as we know the spots $j,\ldots,\pi(i)-1$ are occupied, so $c_i$ will find $a_i$ occupied and by the parking rule will move forward and park at the first available spot which is spot $\pi(i)$. 
This establishes Fact 2 as the number of possibilities for $a_i$ is precisely the length of the longest subsequence $s_j\cdots s_{\pi(i)}$ such that $s_t\leq s_{\pi(i)}$, for all $j\leq t\leq \pi(i)$.
\end{proof}

Summing over the fiber for each permutation gives the number of parking functions.
\begin{corollary}[Exercise 5.49(d,e)~\cite{EC2}]\label{cor:EnumeratingPF}
$$\sum_{\sigma \in S_n}{\bigg(\prod_{i = 1}^n \len{}{i}{\sigma} \bigg)} = |PF_n| = (n+1)^{n-1}$$
where for each $\sigma = s_1\cdots s_n \in S_n$, $\len{}{i}{\sigma}$ is the length of the longest subsequence $s_j\ldots s_i$ of $\sigma$ such that $s_t\leq i$ for all $j\leq t\leq i$.  
\end{corollary}

\begin{remark} 
Recently, Sanyal and Drohla~\cite{SanyalDrohla} gave a different formula for counting parking functions, which they then related to binary trees.
\end{remark}

\subsection{ \texorpdfstring{$k$}{k}-Naples parking functions}\label{subsect:kNaplesandPermutations}
We now generalize Corollary~\ref{cor:EnumeratingPF} to $k$-Naples parking functions. To start, we generalize the map $\varphi$.

\begin{definition}\label{def:varphi_k}
Given a $k$-Naples parking function, consider the function $\varphi_k: PF_{n,k} \to S_n$ given by mapping a $k$-Naples parking function $\alpha \in PF_{n,k}$ to the permutation denoting the position in which the cars park under $\alpha$ using the \textbf{$k$-Naples parking rule}. 
That is, given $\alpha = (a_1,a_2,\ldots, a_n) \in PF_n$, $\varphi(\alpha) = s_1s_2\cdots s_n$, where the spot $i$ is occupied by the ${s_i}^{\text{th}}$ car.
\end{definition}

As before, $\varphi_k$ is the identity on $S_n\subset PF_{n,k}$ for all $k$.
Moreover, from Definition \ref{def:varphi} we note that $\varphi = \varphi_0$. 
However, $\varphi_k$  is not an extension of $\varphi$ for $k>0$. 
Moreover, we note that the order in which cars park under $\varphi_k$ might change when allowing the cars to back up one additional spot, as they might find that spot available rather than moving forward, in which case $\varphi_{k+1}$ would be different than $\varphi_{k}$.
We illustrate this in the following example. 

\begin{example}
Consider $\alpha = (4,2,2,4,1)$.
By Example~\ref{example:permutations1}, we know that $\alpha$ is a parking function. 
In fact, $\alpha$ is also a $1$-Naples parking function. 
Under the $1$-Naples parking rule, the cars park in the following order:
$$\begin{array}{c|c|c|c|c}
\text{spot } 1 & \text{spot } 2 & \text{spot } 3 & \text{spot } 4 & \text{spot } 5\\\hline
c_3 & c_2 & c_4 & c_1 & c_5 
\end{array} 
$$
Therefore, $\varphi_1(\alpha) = 32415$. 
From Example~\ref{example:permutations1}, we note that $\varphi_1(\alpha)\neq \varphi(\alpha)$.
\end{example}

Just as $\len{}{i}{\sigma}$ was defined in terms of the subsequences to the left (see Proposition~\ref{prop:permutation}), we must now extend its definition and consider also subsequences to the right. This allows us to account for the backwards movement of the cars arising from the $k$-Naples parking rule.  

\begin{definition}\label{def:subsequences}
Let $\sigma = s_1\cdots s_n\in S_n$ be a permutation. For each $1\leq i\leq n$, let $\ileft{k}{i}{\sigma}$ be the length of longest subsequence $s_j\cdots s_{i-1}$ of $\sigma$ such that $s_t < s_i$, for all $j\leq t< i$ and let $\iright{k}{i}{\sigma}$ be the length of longest subsequence $s_i\cdots s_r$ of $\sigma$ such that $r\leq i+k$ and $s_t\leq s_i$ for all $i\leq t\leq r$. If these subsequences are empty, $\ileft{k}{i}{\sigma} =0$ or $\iright{k}{i}{\sigma}=0$, respectively. 
Let $\len{k}{i}{\sigma}$ be the function defined by
\[\len{k}{i}{\sigma}=
\begin{cases}
    \ileft{k}{i}{\sigma} + \iright{k}{i}{\sigma} & \text{if $\ileft{k}{i}{\sigma} = i-1$} \\
    \max(\ileft{k}{i}{\sigma}-k, 0) + \iright{k}{i}{\sigma} & \text{if $\ileft{k}{i}{\sigma} < i-1$.}
\end{cases}\]
\end{definition}

Before working through an example illustrating Definition~\ref{def:subsequences} we remind the reader that the parameter $k$ is the same as the parameter defining the $k$-Naples parking functions. Moreover, we can also restate the definition of $\len{k}{i}{\sigma}$ in terms of the $\len{}{i}{\sigma}$ and the reversal map $\text{rev}(\sigma)$ which sends $\sigma_i$ to $\sigma_{n-i+1}$.

\begin{remark}
The values \textnormal{$\ileft{k}{i}{\sigma}$} and \textnormal{$\iright{k}{i}{\sigma}$} can be expressed in terms of $\len{}{i}{\sigma}$.  In particular,
\begin{itemize}
    \item[] \textnormal{$\ileft{k}{i}{\sigma}=\len{}{i}{\sigma}-1$}
    \item[] \textnormal{$\iright{k}{i}{\sigma}=\min(k+1,\len{}{i}{\textnormal{rev}(\sigma)})$}.
\end{itemize}
Therefore, $\len{k}{i}{\sigma}$ can also be expressed in terms of $\len{}{i}{\sigma}$.  In particular,
\[\len{k}{i}{\sigma}=
\begin{cases}
    \len{}{i}{\sigma}-1 + \min(k+1,\len{}{i}{\textnormal{rev}(\sigma)} & \text{if $\len{}{i}{\sigma}  = i$} \\
    \max(\len{}{i}{\sigma}-1-k, 0) + \min(k+1,\len{}{i}{\textnormal{rev}(\sigma)} & \text{if $\len{}{i}{\sigma} < i$}.
\end{cases}\]
\end{remark}

\begin{example}\label{ex:details}
Let $n=5$ and $k=2$, and consider the permutation $\sigma = 51423\in S_5$. Then, $\ileft{2}{1}{51423}= \ileft{2}{2}{51423}=\ileft{2}{4}{51423}=0$ because the corresponding subsequences are all the empty subsequence. Moreover, $\ileft{2}{3}{51423}=\ileft{2}{5}{51423} = 1$, corresponding to the subsequences $1$ and $2$, respectively. 

Similarly, $\iright{2}{1}{51423}= \iright{2}{3}{51423} =3$, corresponding to the subsequences $514$ and $423$, respectively, and $\iright{2}{2}{51423}= \iright{2}{4}{51423}= \iright{2}{5}{51423} = 1$, corresponding to the subsequences $1$, $2$ and $3$, respectively.

Therefore, we get the following values for $\len{k}{i}{51423}$: $\len{2}{1}{51423} = 3$, $\len{2}{2}{51423} = 1$, $\len{2}{3}{51423}= \max(1-2,0) + 3 = 3 $, $\len{2}{4}{51423}= 1$, and $\len{2}{5}{51423} = \max(1-2,0)+1 = 1 $.

A straight forward computation establishes that there are $9$ distinct 2-Naples parking functions of length 5 whose image under $\varphi_2$ is the permutation 51423. We list them below,
\[
\begin{array}{ccccccccc}
24531 & 24532 & 24533 & 24541 & 24542 & 24543 & 24551 & 24552 & 24553.
\end{array}
\]
In our next result we establish that the count of these nine distinct $2$-Naples parking functions arises as the product of the values $\len{k}{i}{\sigma}$, which from above we note is given by $3\times 1\times 3\times 1\times 1=9$.
\end{example}

We are now ready to state the main result.
\begin{theorem}\label{thm:permutations-k}
Let $\sigma \in S_n$ be a permutation. Then,
$$|\varphi_k^{-1}(\sigma)| = \prod_{i=1}^n\len{k}{i}{\sigma}.$$
\end{theorem}
In particular, for $k=0$, we obtain the result in Proposition~\ref{prop:permutation}.

\begin{proof}
Let $\alpha = (a_1, \ldots, a_n)\in \varphi_k^{-1}(\sigma)$ and take $\pi \in S_n$ such that $s_{\pi(i)} = i$. That is, $c_i$ parks in the parking spot $\pi(i)$.
To establish the result it suffices to show the following:

\begin{quote}
    \textbf{Claim:} \textit{For each $1\leq i\leq n$, $\len{k}{\pi(i)}{\sigma}$ is the number of possible preferred parking spots $a_i$ which allows $c_i$ to park in spot $\pi(i)$.}
\end{quote}

Proving this claim then implies the result as we take the product over all possibilities for $a_1,\ldots, a_n$. Let us prove the claim.

Given $i$, let $s_j\cdots s_{\pi(i) - 1}$ be the longest subsequence of $\sigma = s_1\cdots s_n$ such that $s_t< s_{\pi(i)}$ for all $j\leq t < \pi(i)$ and $s_{\pi(i)}\cdots s_{r}$ the longest subsequence of $\sigma$ such that $s_t\leq s_{\pi(i)}$ for all $\pi(i)\leq t\leq r\leq \pi(i)+k$. These are  exactly the subsequences whose length define the values $\ileft{k}{\pi(i)}{\sigma}$ and $\iright{k}{\pi(i)}{\sigma}$ in the definition of $\len{k}{\pi(i)}{\sigma}$ (Definition~\ref{def:subsequences}).

Our next step is to show that these are precisely all the possibilities for the parking preferences $a_i$ of $c_i$.
By definition, $c_i$ parks in the spot numbered $\pi(i)$. 
Note that just before $c_i$ parks,
the parking spots {$\pi(1), \ldots, \pi(i-1)$} have been occupied by $c_1,\ldots, c_{i-1}$ respectively. 
By the \textit{$k$-Naples parking rule}, we consider three cases depending on where $c_i$ parks:
\begin{enumerate}[leftmargin=.75in]
    \item[\textbf{Case 1}:] $c_i$ parks in its preferred spot $a_i$, or
    \item[\textbf{Case 2}:] $c_i$ parks in a spot before its preferred spot $a_i$, or
    \item[\textbf{Case 3}:] $c_i$ parks in a spot after its preferred spot $a_i$.
\end{enumerate}
By definition, the count for $\iright{k}{\pi(i)}{\sigma}$ involves only the cases (1) and (2), whereas the count for $\ileft{k}{\pi(i)}{\sigma}$ involves only case (3). Note also that the contribution of $\iright{k}{\pi(i)}{\sigma}$ to $\len{k}{\pi(i)}{\sigma}$ is the same independently of the value of $\ileft{k}{\pi(i)}{\sigma}$. 

{We now consider the implications arising from each of the possible cases defined above.}
\begin{enumerate}[leftmargin=.75in]
    \item[\textbf{Case 1}:] {Assume} $c_i$ parks in its preferred spot $a_i$.\\
    In this case, $a_i=\pi(i)$ and this contributes to the value of $\iright{k}{\pi(i)}{\sigma}$. 
    \item[\textbf{Case 2}:] {Assume} $c_i$ parks in a spot before its preferred spot $a_i$. \\
    In this case, $c_i$ tries to park in the spot $a_i$, with $\pi(i)<a_i\leq \pi(i)+k$. However, the parking spots  $\pi(i)+1,\ldots, a_i$ are all occupied, and $c_i$ backs up until spot $\pi(i)$. Now, this implies that  $\{\pi(i)+1,\ldots, a_i\}$ is a subset of the previously occupied spots $\{\pi(1),\ldots,\pi(i-1)\}$. 
    Thus, $\{s_{\pi(i)+1},\dots, s_{a_i}\}\subset\{s_{\pi(1)},\dots, s_{\pi(i-1)}\} =\{1,\ldots,i-1\}$.  
    Therefore, for any $t$ satisfying $\pi(i)<t\leq a_i$, there exists $r\in \{1,\ldots, i-1\}$ such that $s_t = s_{\pi(r)} = r \leq i = s_{\pi(i)}$ and this case contributes to the value of $\iright{k}{\pi(i)}{\sigma}$. 
    
    \item[\textbf{Case 3}:] {Assume} $c_i$ parks in a spot after its preferred spot $a_i$. \\
    In this case, $c_i$ tries to park in $a_i$, with $a_i<\pi(i)$, which is occupied. In this case, the contribution to $\len{k}{\pi(i)}{\sigma}$ depends on the value of $\ileft{k}{\pi(i)}{\sigma}$. We look at each possible case:
    \begin{enumerate}[leftmargin=.5in]
        \item[\textbf{Subcase 3a}:]
        If $\ileft{k}{\pi(i)}{\sigma}=\pi(i)-1$, the spots before $\pi(i)$ are all occupied. That is, there are no empty spots available before $a_i$ and between $a_i$ and $\pi(i)$. Then, $\{a_i,\ldots, \pi(i)-1\}$ is a subset of the previously occupied spots $\{\pi(1), \ldots, \pi(i-1)\}$. Thus, $\{s_{a_i},\ldots, s_{\pi(i)-1}\} \subset \{s_{\pi(1)}, \ldots, s_{\pi(i-1)}\}=\{1,\ldots,i-1\}$. Therefore, for any $t$ satisfying $a_i\leq t<\pi(i)$, there exists $r\in \{1,\ldots, i-1\}$ such that $s_t = s_{\pi(r)} = r < i = s_{\pi(i)}$ and this contributes to the value of $\ileft{k}{\pi(i)}{\sigma}$.
        
        \item[\textbf{Subcase 3b}:] 
        If $\ileft{k}{\pi(i)}{\sigma} < \pi(i)-1$, then there exists an empty spot before $\pi(i)$. In fact, this empty spot has to be before $a_i-k$. Otherwise, {by the $k$-Naples parking rule,} $c_i$ may be able to back up into a parking spot before $a_i${, which would yield a contradiction.} Therefore, there are  $\max(\ileft{k}{i}{\sigma}-k,0)$ spots available since $\ileft{k}{i}{\sigma}-k$ can be negative. 
    \end{enumerate}
\end{enumerate}

In Cases 1-3, either $s_t < s_{\pi(i)}$, for all $t$ satisfying $a_i\leq t< \pi(i)$, or $s_t\leq s_{\pi(i)}$, for all $t$ satisfying $\pi(i)\leq t \leq a_i\leq \pi(i) + k$. Therefore, these give necessary conditions for $a_i$ in order for $c_i$ to park in spot $\pi(i)$, proving our claim.

As in the proof of Proposition \ref{prop:permutation}, it is clear that any index of the subsequence corresponding to $\ell_k(\pi(i);\sigma)$ is a parking preference $a_i$ which will cause $c_i$ to park at the spot $\pi(i)$, providing the sufficient condition, and thus completing the proof of the theorem.
\end{proof}

As a consequence of Theorem \ref{thm:permutations-k}, we give a new expression for the number of $k$-Naples parking functions.

\begin{theorem}\label{thm:pfnk}
For all $n\geq q$ and $0\leq k\leq n-1$, 
$$|PF_{n,k}| = \sum_{\sigma \in S_n}{\bigg(\prod_{i = 1}^n \len{k}{i}{\sigma} \bigg)}.$$
\end{theorem}

This result provides a new formula for the enumeration of $k$-Naples parking functions.
Compared to Theorem~\ref{thm:RecEnumNPF}, this formula is not recursive and is presented in terms of a very familiar set of combinatorial objects, namely, permutations.

\section{A logarithmic generating function on the fibers of $\varphi$}\label{section:GFvarphi} 

Next we give a recurrence relation for the fibers of $\varphi$. That is, we present a formula for counting the number of permutations whose fiber under $\varphi$ has a given size.

Let us denote by $F_n(q)$ the generating function of the fiber of $\varphi$ for each $n\geq 0$.  
That is, 
$$F_n(q) = \sum_{i\geq 1} c_{n,i}q^i,$$
where the coefficient $c_{n,i}$ counts the number of permutations in $S_n$ whose fiber under $\varphi$ has size $i$. Note that $F_0(q)=q$ since there is trivially one permutation of length 0 and one parking function of length 0 mapping to that permutation. 

The first five polynomials are given in~\cite{FindStat}:
\begin{eqnarray*}
    F_1(q) &=& q, \\
    F_2(q) &=& q + q^2, \\
    F_3(q) &=& q + 3q^2 + q^3 + q^6, \\
    F_4(q) &=& q + 6q^2 + 4q^3 + 4q^4 + 4q^6 + 3q^8 + q^{12} + q^{24}, \\
   F_5(q) &=& q + 10q^2 + 10q^3 + 20q^4 + q^5 + 20q^6 + 15q^8 + 6q^{10} + 15q^{12}\\ &&  + 4q^{15} + 4 q^{20}+5 q^{24} + 4q^{30}+3 q^{40}+q^{60}+q^{120}. 
\end{eqnarray*}
Defining $c_{n,i} = 0$ if $n < 0$, we next show how the coefficients $c_{n,i}$ can be computed recursively.
\begin{proposition}\label{prop:recursiveformulacoeff}
For $1\leq i\leq n!$,
$$c_{n,i} = \sum_{d|i}\left[ \binom{n-1}{d-1}\left(\sum_{j\left|\frac{i}{d}\right.} c_{d-1,j} c_{n-d, \frac{i}{dj}} \right)\right],$$
and $c_{n,i}=0$, otherwise.
\end{proposition}
\begin{proof}
Given $i$, we want the number of elements $\sigma \in S_n$ whose fiber has size $i$, i.e., $|\varphi^{-1}(\sigma)|=i$. 
Let $\sigma = s_1\ldots s_n \in S_n$ be one of those permutations, and let $d$ be the index such that $s_d = n$.
We split $\sigma$ into three pieces (allowing for the case where some may be empty)
$$
\sigma = \underbrace{s_1\cdots s_{d-1}}_{s^\prime} s_d \underbrace{s_{d+1}\cdots s_n}_{s^{\prime\prime}}.
$$
Assuming neither $s^\prime$ nor $s^{\prime\prime}$ are empty, we look at $s^\prime$ and $s^{\prime\prime}$ as permutations of $d-1:=\ell(s^\prime)$ and $n-d:=\ell(s^{\prime\prime})$, 
by reducing the values to $1,\ldots,d-1$ for $s^\prime$ and $1,\ldots,n-d$ for $s^{\prime\prime}$ while preserving the relative order of the entries.
Let us denote by $\sigma^\prime$ and $\sigma^{\prime\prime}$ those permutations (respectively).
Therefore, we have that
$$ i = |\varphi^{-1}(\sigma) | = |\varphi^{-1}(\sigma^{\prime} )|\cdot d\cdot |\varphi^{-1}(\sigma^{\prime\prime} )|. $$
On the other hand, if $s^\prime$ or $s^{\prime\prime}$ is empty, then they are the empty permutation in which case either $|\varphi^{-1}(\sigma^{\prime} )| = 1$ or $|\varphi^{-1}(\sigma^{\prime\prime} )|=1$.

Since $|\varphi^{-1}(\sigma^\prime)|$ and $|\varphi^{-1}(\sigma^{\prime\prime})|$ are integers, $d$ must be a factor of $i$. In fact, for each possible $d$, there are $\displaystyle{\binom{n-1}{d-1}}$ ways to choose the first $d$ cars (unordered). The ordering of these $d-1$ cars is uniquely determined by $\sigma^\prime$. Thus,
\[c_{n,i} = \sum_{d|i} \left[\binom{n-1}{d-1} \left(\sum_{j\left|\frac{i}{d}\right.} c_{d-1,j} c_{n-d, \frac{i}{dj}} \right)\right].\qedhere\]
\end{proof}

Consider now the \textit{logarithmic generating function} defined for $n\geq 1$ by 
$$G_n(q) = \sum_{i=1}^{n!} c_{n,i}q^{\ln i}.$$ This generating function allows us to manipulate exponents easily, obtaining the following recursive formula.
\begin{proposition}\label{prop:generatingfunction}
For $n\geq 1$,
$$G_n(q) = \sum_{i=0}^{n-1} \binom{n-1}{i}q^{\ln(i+1)}G_i(q)G_{n-1-i}(q).$$
\end{proposition}
\begin{proof}
By Proposition~\ref{prop:recursiveformulacoeff}, we have that
\begin{align*}
    G_n(q) &= \sum_{i=1}^{n!} c_{n,i}q^{\ln i} = \sum_{i|n!} c_{n,i}q^{\ln i} 
    = \sum_{i|n!} \sum_{d|i} q^{\ln d} \binom{n-1}{d-1} \left(\sum_{j\left|\frac{i}{d}\right.} c_{d-1,j}q^{\ln j} c_{n-d, \frac{i}{dj}}q^{\ln \frac{i}{dj}} \right) \\ 
    &= \sum_{i|n!}\sum_{d|i} q^{\ln d} \binom{n-1}{d-1} \left(\sum_{\substack{r \text{ s.t.} \\  d|r, r|i}} c_{d-1,\frac{r}{d}} q^{\ln \frac{r}{d}} c_{n-d,\frac{i}{r}} q^{\ln \frac{i}{r}}\right).
\end{align*}
Now, since $\displaystyle{\binom{n-1}{d-1}=0}$ for $d>n$, we can restrict the summation over $d$ and also change the order of the summations over $d$ and over $i$. Thus, we have that
    \begin{align*}
    G_n(q) &=   \sum_{i|n!} \sum_{\substack{d|i\\ d\leq n}} q^{\ln d} \binom{n-1}{d-1} \left(\sum_{\substack{r \text{ s.t.} \\  d|r, r|i}} c_{d-1,\frac{r}{d}} q^{\ln \frac{r}{d}} c_{n-d,\frac{i}{r}} q^{\ln \frac{i}{r}}\right) \\
    &= \sum_{d=1}^n q^{\ln d} \binom{n-1}{d-1} \left(\sum_{\substack{i \text{ s.t.} \\  d|i, i|n!}}\sum_{\substack{r \text{ s.t.} \\  d|r, r|i}}  c_{d-1,\frac{r}{d}} q^{\ln \frac{r}{d}} c_{n-d,\frac{i}{r}} q^{\ln \frac{i}{r}}\right)\\
    &= \sum_{d=1}^n q^{\ln d} \binom{n-1}{d-1} \left(\sum_{i'\left|\frac{n!}{d}\right.} \sum_{r'|i'} c_{d-1,r'} q^{\ln r'} c_{n-d,\frac{i'}{r'}} q^{\ln \frac{i'}{r'}}\right),
\end{align*}
where in the last equality we rewrite the summation with a new index. Again, we have that $c_{d-1,r} = 0$ if $r'\nmid (d-1)!$, and similarly for $\dfrac{i'}{r'}$. 
Therefore, 
\begin{align*}
    G_n(q) &= \sum_{d=1}^n q^{\ln d} \binom{n-1}{d-1} \left(\sum_{r^\prime} \sum_{i^\prime} c_{d-1,r'} q^{\ln r'} c_{n-d,\frac{i'}{r'}} q^{\ln \frac{i'}{r'}}\right),
\end{align*}
where the second summation is over the $r^\prime$ such that $r'\left|\dfrac{n!}{d}\right.$ and $r'| (d-1)!$, and the third summation is over $i^\prime$ such that $\dfrac{i'}{r'}\left\vert\dfrac{n!}{r' d}\right.$ and $\left.\dfrac{i'}{r'}\right| (n-d)!$,
that is,
\begin{align*}
    G_n(q) &= \sum_{d=1}^n q^{\ln d} \binom{n-1}{d-1} \left(\sum_{r'| (d-1)!} c_{d-1,r'} q^{\ln r'} \sum_{\left.\frac{i'}{r'}\right\vert(n-d)!}  c_{n-d,\frac{i'}{r'}} q^{\ln \frac{i'}{r'}}\right).
\end{align*}
Now, the summation over $i^\prime$ is exactly the definition of $G_{n-d}(q)$ and the summation over $r^\prime$ is the definition of $G_{d-1}(q)$.
Therefore, 
\[
G_n(q) = \sum_{d=1}^{n} \binom{n-1}{d-1}q^{\ln(d)} G_{d-1}(q)G_{n-d}(q).
\]
Writing $i=d-1$, we have that
\[
G_n(q) = \sum_{i=0}^{n-1} \binom{n-1}{i}q^{\ln(i+1)} G_i(q)G_{n-1-i}(q).\qedhere
\]
\end{proof}

\section{Finding \texorpdfstring{$q$}{q}-analogues}\label{section:qanalogs}
Given a non-negative integer $n$, we denote its $q$-analog by 
$$ [n]_q =\lim_{q\rightarrow 1} \dfrac{1-q^n}{1-q} =  1+q+q^2+\cdots +q^{n-1}.$$ 
With this $q$-analog, one may define $q$-analogs of almost any formula. 
For instance, the $q$-factorial or the $q$-binomial coefficients.
In general, we denote by $[\star]_q$ the $q$-analog of the formula $\star$ by substituting any number appearing in $\star$ by its $q$-analog.

Together with $q$-analogs, we also have $q$-statistics, which allow us to understand combinatorially the exponents of $q$ appearing in the $q$-analog formulas, mostly polynomials in $q$, once they are expressed as summations over the exponents of $q$. \textit{FindStat}~\cite{FindStat} is a great source to find many statistics as well as good references. 

The $q$-analog appears naturally in several contexts in algebraic combinatorics, and in particular, in the the framework of parking functions and diagonal harmonics.
In~\cite{dh}, Loehr presents several identities involving $q$-analogs in the framework of diagonal harmonics.
We refer to~\cite{dh} for more details on the diagonal harmonics. 
One of the $q$-statistics that appears in our setting is the $\texttt{area}$ of a parking function, which we now define. 

\begin{definition}\label{def:area}
Given a parking function $\alpha\in PF_n$, let $P$ be the labeled Dyck path associated to it. We define $\texttt{area}(\alpha)$ as the number of complete lattice squares between $P$ and the line $y=x$. Equivalently, the \texttt{area} of the parking function $\alpha$ can be computed as
\[\texttt{area}(\alpha)=\sum_{i=1}^n \left(n-i-a_i+1\right).\]
\end{definition}

Our first result provides the distribution of the area statistic over the fiber $\varphi^{-1}(\sigma)$, for any permutation $\sigma$, by looking at the $q$-analog of Corollary~\ref{cor:EnumeratingPF}.
\begin{proposition}\label{prop:classical-q}
For any permutation $\sigma\in S_n$,
\begin{align}\label{eq:classical-q}
    \sum_{p\in \varphi^{-1}(\sigma)}q^{\textnormal{\texttt{area}}(p)}=\prod_{i=1}^n [\len{}{i}{\sigma}]_q.
\end{align}
\end{proposition}
We present two proofs of this result: one bijective and one inductive. 
The bijective proof is generalized for $k$-Naples parking functions in Section~\ref{subsec:kNaplesq}. 
\begin{proof}[Bijective proof]
Using notation from~\cite[Fermionic formula]{dh}, the right-hand side of \eqref{eq:classical-q} gives a generating function for objects $(u_1,u_2,\ldots, u_n)$ with $u_i<\len{}{i}{\sigma}$, where $q^m$ corresponds to such an object with $\sum u_i=m$.
Call the set of these objects $\mathcal{I}$ and define $\texttt{stat}((u_i))=\sum u_i$.  Immediately
\[\prod_{i=1}^n \left[ \len{}{i}{\sigma}\right]_q
= \sum_{(u_i)\in \mathcal{I}} q^{\texttt{stat}((u_i))},\]
so it remains to find a bijection $f:\mathcal{I}\to PF_{n}$ with $\texttt{stat}=\texttt{area}\circ f$.  
Consider arbitrary $(u_i)\in\mathcal{I}$ and choose $\pi$ based on $\sigma$ as in the proof of Proposition~\ref{prop:permutation}.
Then, we define the map
\[
\begin{array}{rccl}
f:& \mathcal{I} & \longrightarrow & PF_{n} \\
& (u_i) & \longmapsto & (\pi_1-u_1,\pi_2-u_2, \ldots,\pi_n-u_n).
\end{array}
\]
We claim that $f$ is a bijection.  It is well-defined because $\pi$ is a parking function, and weakly reducing all terms of a parking function gives a parking function.  Injectivity follows since $\pi$ can be recovered as the outcome of the image, and then $u_i=\pi_i-a_i$.  The function is surjective by the casework from Proposition~\ref{prop:permutation}. Since swapping columns preserves area, the permutation $\pi$ can be rearranged to the permutation $n\ n-1\cdots 21$ which has no area. Now reducing each $\pi_i$ by $u_i$ adds $u_i$ to the area, so $\texttt{area}\circ f=\sum u_i$ as desired.
\end{proof}

\begin{proof}[Inductive proof] 
For $n=1$, there is only one permutation $\sigma = (1)$, for which $\varphi^{-1}(\sigma) = PF_1 = \{(1)\}$. Therefore, 
\[\sum_{p\in \varphi^{-1}(\sigma)}q^{\texttt{area}(p)}= q^{\texttt{area}(1)} = q^0  = 1 = [\len{}{1}{1}]_q = \prod_{i=1}^n [\len{}{i}{\sigma}]_q.\] 

Now, suppose that the statement is true for $1,\ldots, n-1$. 
Let $\sigma = s_1\cdots s_n \in S_n$. Let $d$ be the index such that $s_d = n$. As in the proof of Proposition~\ref{prop:generatingfunction}, we split $\sigma$ into three pieces 
$$
\sigma = \underbrace{s_1\cdots s_{d-1}}_{s^\prime} s_d \underbrace{s_{d+1}\cdots s_n}_{s^{\prime\prime}},$$
and we consider the permutations $\sigma^\prime \in S_{d-1}$ and $\sigma^{\prime\prime}\in S_{n-d}$.

Since the value of $\len{}{i}{\sigma}$ only depends on the relative order of the numbers, $\len{}{i}{\sigma} = \len{}{i}{\sigma^\prime}$ for all $1\leq i\leq d-1$, $\len{}{i}{\sigma}=\len{}{i-d}{\sigma^{\prime\prime}}$ for all $d+1\leq i\leq n$, and $\len{}{d}{\sigma} = d$. 
Thus by induction hypothesis,
\begin{align*}
    \prod_{i=1}^n[\len{}{i}{\sigma}]_q &= \left(\prod_{i=1}^{d-1} [\len{}{i}{\sigma^\prime}]_q \right) [d]_q \left(\prod_{i=1}^{n-d}[\len{}{i}{\sigma^{\prime\prime}}]_q\right)\\
    &= [d]_q\left(\sum_{p^\prime\in\varphi^{-1}(\sigma')} q^{\texttt{area}(p^\prime)}\right) \left(\sum_{p^{\prime\prime}\in \varphi^{-1}(\sigma^{\prime\prime})} q^{\texttt{area}(p^{\prime\prime})}\right).
\end{align*}

Let $t_1,\ldots, t_{d-1}$ and $t_{d+1},\ldots, t_n$ be the increasing rearrangements of $s_1,\ldots, s_{d-1}$ and $s_{d+1},\ldots, s_n$, respectively. For the RHS in the formula of Proposition~\ref{prop:classical-q}, since the $n^{\text{th}}$ car $c_n$ parks at spot $d$, the parking functions $(a_1,\ldots, a_n)$ which map to $\sigma = s_1\cdots s_n$ are precisely those such that $(a_{t_1},\ldots, a_{t_{d-1}})\in\varphi^{-1}(\sigma^\prime)$, $(a_{t_{d+1}},\ldots, a_{t_n})\in\varphi^{-1}(\sigma^{\prime\prime})$, and $a_n \in \{1,\ldots, d\}$. 
Moreover, the area of a parking function is the area of its decreasing rearrangement, and so
$$\texttt{area}(a_1,\ldots, a_n) = \texttt{area}(a_{t_1},\ldots, a_{t_{d-1}}) + \texttt{area}(a_{t_{d+1}},\ldots, a_{t_n}) + d-a_n.$$
Therefore, we have
\begin{align*}
    \sum_{p\in\varphi^{-1}(\sigma)} q^{\texttt{area}(p)} &= \sum_{r=1}^d \sum_{p^\prime \in\varphi^{-1}(\sigma^\prime)} \sum_{p^{\prime\prime}\in\varphi^{-1}(\sigma^{\prime\prime})} q^{\texttt{area}(p^{\prime})} q^{\texttt{area}(p^{\prime\prime})}q^{d-r}\\
    &= \Bigg(1+q+\cdots+ q^{d-1}\Bigg) \Bigg(\sum_{p^\prime\in \varphi^{-1}(\sigma^{\prime})} q^{\texttt{area}(p^{\prime})}\Bigg)
    \Bigg(\sum_{p^{\prime\prime}\in\varphi^{-1}(\sigma^{\prime\prime})}q^{\texttt{area}(p^{\prime\prime})}\Bigg)\\
    &= [d]_q\Bigg(\sum_{p^{\prime}\in\varphi^{-1}(\sigma^\prime)} q^{\texttt{area}(p^\prime)}\Bigg) \Bigg(\sum_{p^{\prime\prime}\in \varphi^{-1}(\sigma^{\prime\prime})} q^{\texttt{area}(p^{\prime\prime})}\Bigg)
    = \prod_{i=1}^n[\len{}{i}{\sigma}]_q,
\end{align*}
which completes the induction.
\end{proof}

\begin{corollary}
\label{cor:classical-q}
For each $n$,
\[
\sum_{\sigma\in S_n}\prod_{i=1}^n [\len{}{i}{\sigma}]_q = 
\sum_{p\in PF_n}q^{\texttt{area}(p)}
\]
\end{corollary}
\begin{proof}
By Proposition~\ref{prop:classical-q}, summing over all fibers, we have that
\[
\sum_{\sigma\in S_n}\prod_{i=1}^n [\len{}{i}{\sigma}]_q =
\sum_{\sigma\in S_n} \sum_{p\in \varphi^{-1}(\sigma)}q^{\texttt{area}(p)} = 
\sum_{p\in PF_n}q^{\texttt{area}(p)}.\qedhere\] 
\end{proof}

By~\cite[Univariate symmetry]{dh}, we have the following chain of identities.
\begin{corollary}
\[ \sum_{\sigma\in S_n}\left(\prod_{i=1}^n \left[ \len{}{i}{\sigma}\right]_q
\right) = \sum_{p\in PF_n} q^{\texttt{area}(p)} = \sum_{p\in PF_n} q^{\texttt{dinv}(p)} = \sum_{p\in PF_n} q^{\texttt{pmaj}(p)}\]
\noindent where \texttt{dinv} and \texttt{pmaj} are other well-known statistics on parking functions.
\end{corollary}

In fact, following the $q,t$-statistics presented in~\cite{dh}, we arrive at the following result.
\begin{corollary}
For any statistic $\texttt{stat}$ on permutations, 
$$\sum_{\sigma\in S_n} t^{\texttt{stat}(\sigma)} \left(\prod_{i=1}^n \left[ \len{}{i}{\sigma}\right]_q \right) = \sum_{p\in PF_n}q^{\texttt{area}(p)} t^{\texttt{stat}(\varphi(p))}.$$
\end{corollary}

\subsection{The \texorpdfstring{$k$}{k}-Naples area statistics}\label{subsec:kNaplesq}
Next we generalize Proposition~\ref{prop:classical-q} to $k$-Naples parking functions.
We begin by introducing an \textit{area statistic}. 
For parking functions, the area statistic counts the distance between the Dyck path and the main diagonal, see Definition~\ref{def:area}. 
Alternatively, one can think about the area as distance between the lattice path and the highest point among lattice paths in the same fiber of $\varphi$. 
Therefore, our definition for the $k$-Naples area counts the distance between the lattice path given by the $a_i$'s and the highest that path can be while the corresponding parking function still remains in the fiber of $\varphi_k$.

\begin{definition}
Given a parking function $\alpha =(a_1,a_2,\ldots, a_n)$, we define the \textbf{$k$-Naples area} by 
\[
\texttt{area}_k(\alpha)=\sum_{i=1}^n  \left[ n-i+\iright{k}{i}{\sigma}\left(\varphi_k(\alpha)\right)-a_i\right].
\]
\end{definition}

Note that for $k=0$, $\iright{k}{i}{\sigma}=1$ and we recover the definition of area for parking functions.
Bear in mind, however, that $\iright{k}{i}{\sigma}$ also depends on $k$.

\begin{example}
Consider $\alpha=(3,2,2)$ as a $1$-Naples parking function, for which $\varphi_1(\alpha)=321$.  Now $\iright{1}{1}{321}=\iright{1}{2}{321}=2$ and $\iright{1}{3}{321}=1$. Therefore,  $$\texttt{area}_1(\alpha)=(3-1+2-3)+(3-2+2-2)+(3-3+1-2)=1.$$
\end{example}

We now study the distribution of the $\texttt{area}_k$ statistic over the fiber $\varphi_k^{-1}(\sigma)$, where $\sigma$ is a permutation.  We find that it is the $q$-analog of the formula in Theorem \ref{thm:permutations-k}.
\begin{proposition}\label{prop:kq}
For any permutation $\sigma\in S_n$, 
\[\sum_{p\in\varphi_k^{-1}(\sigma)}q^{\textnormal{\texttt{area}}_k(p)}=\prod_{i=1}^n [\len{k}{i}{\sigma}]_q.\] 
\end{proposition}

\begin{proof}
As in the proof of Proposition~\ref{prop:classical-q}, consider the $q$-analogue as a generating function for objects $(u_i)$, this time with $u_i<\len{k}{i}{\sigma}$. 
Again, let $\mathcal{I}$ be the set of these objects, and define a statistic $\texttt{stat}((u_i))=\sum u_i$. 
Now the $q$-analogue is a generating function satisfying

\[ \prod_{i=1}^n \left[ \len{k}{i}{\sigma}\right]_q = \sum_{(u_i)\in\mathcal{I}} q^{\texttt{stat}((u_i))}.\]

It remains to find a bijection $f:\mathcal{I}\to \varphi^{-1}_k(\sigma)$ with $\texttt{stat}=\texttt{area}_k\circ f$. We define the map $f$ as
\[
 f((u_i)) = (\pi_1 + \iright{k}{1}{\sigma} - u_1 - 1,\pi_2 + \iright{k}{2}{\sigma}  - u_2 - 1,\ldots,\pi_n + \iright{k}{n}{\sigma} - u_n - 1),
\]
for $(u_i)\in \mathcal{I}$, where $\pi$ is defined in the proof of Theorem~\ref{thm:permutations-k}.
We claim that $f$ is a bijection. To see that $f$ is well-defined, we start with the permutation $\pi=\pi_1\cdots\pi_n$, which is a parking function. Then, adding $\iright{k}{i}{\sigma} -1$ to $\pi_i$ leads to the highest point that the path could be, which is still a $k$-Naples parking function. Finally, weakly reducing each entry by $u_i$ gives a $k$-Naples parking function in the same fiber of $\varphi_k$. 
Moreover, given $\alpha\in \varphi^{-1}_k(\sigma)$, we have that $f^{-1}(\alpha)$ is defined by $u_i=a_i-\pi_i-\iright{k}{i}{\sigma}+1$. This formula is obtained directly from the definition of $f$ by taking $a_i$ as the $i^{\text{th}}$ entry of $f((u_i))$ and writing $u_i$ in terms of $a_i$, $\pi_i$ and $\iright{k}{i}{\sigma}$, using that the values of $\pi_i$ and $\iright{k}{i}{\sigma}$ are given by $\sigma$. 

Finally, the fact that $f$ is surjective follows from the casework in the proof of Theorem~\ref{thm:permutations-k}. 

Now the $k$-Naples area is 
\begin{align*}
\texttt{area}_k \circ f &= \sum \left[n - i + \iright{k}{i}{\sigma}- (\pi_i + \iright{k}{i}{\sigma}- u_i - 1)\right]\\
&= \sum \left[n-i-\pi_i+u_i+1 \right]= \sum u_i = \texttt{stat}((u_i)).\qedhere
\end{align*}
\end{proof}
Taking the $q$-analogues of the $\len{k}{i}{\sigma}$ in Theorem~\ref{thm:pfnk} gives the distribution of the $\texttt{area}_k$ statistic over the $k$-Naples parking functions.
\begin{proposition}
For $n\geq 1$, 
\[\sum_{\sigma\in S_n}\prod_{i=1}^n [\len{k}{i}{\sigma}]_q = \sum_{p\in PF_{n,k}}q^{\textnormal{\texttt{area}$_k$(p)}}.\] 
\end{proposition}
\begin{proof}
By Proposition~\ref{prop:kq}, summing over all fibers gives
\[\sum_{p\in PF_{n,k}}q^{\texttt{area}(p)}=\sum_{\sigma\in S_n} \sum_{p\in \varphi_k^{-1}(\sigma)}q^{\texttt{$\texttt{area}_\text{k}$}(p)}=\sum_{\sigma\in S_n}\prod_{i=1}^n [\len{k}{i}{\sigma}]_q.\qedhere\] 
\end{proof}

We conclude with Table~\ref{table:kNaplesarea}, where we list values of the $k$-Naples area for some values of $k$ and~$n$. 
\begin{table}[ht]
 \begin{tabular}{l | l | l} 
 $n$ & $k$ & Distribution of $\texttt{area}_k$ over $PF_{n,k}$ \\
 \hline
 $1$ & $0$ & $1$\\[0.1in]
 $2$ & $0$ & $q+2$ \\
 $2$ & $1$ & $2q+2$ \\[0.1in]
 $3$ & $0$ & $q^3+3q^2+6q+6$ \\
 $3$ & $1$ & $2q^3+7q^2+9q+6$ \\ 
 $3$ & $2$ & $3q^3+9q^2+9q+6$\\[0.1in]
 $4$ & $0$ & $q^6 + 4q^5 + 10q^4 + 20q^3 + 30q^2 + 36q + 24$\\
$4$ & $1$ & $2q^6 + 9q^5 + 24q^4 + 41q^3 + 53q^2 + 50q + 24$\\
$4$ & $2$ & $3q^6 + 13q^5 + 34q^4 + 58q^3 + 60q^2 + 48q + 24$\\
$4$ & $3$ & $4q^6 + 16q^5 + 40q^4 + 64q^3 + 60q^2 + 48q + 24$\\[0.1in]
$5$ & $0$ & $q^{10} + 5q^9 + 15q^8 + 35q^7 + 70q^6 + 120q^5 + 180q^4 + 240q^3 + 270q^2 + 240q + 120$\\
$5$ & $1$ & $2q^{10} + 11q^9 + 35q^8 + 84q^7 + 165q^6 + 263q^5 + 361q^4 + 429q^3 + 435q^2 + 320q + 120$\\
$5$ & $2$ & $3q^{10} + 16q^9 + 50q^8 + 121q^7 + 238q^6 + 384q^5 + 502q^4 + 529q^3 + 462q^2 + 306q + 120$\\
$5$ & $3$ & $4q^{10} + 21q^9 + 65q^8 + 155q^7 + 295q^6 + 464q^5 + 576q^4 + 550q^3 + 450q^2 + 300q + 120$\\
$5$ & $4$ & $5q^{10} + 25q^9 + 75q^8 + 175q^7 + 325q^6 + 500q^5 + 600q^4 + 550q^3 + 450q^2 + 300q + 120$\\
\hline
 \end{tabular}
 \caption{Values of the $k$-Naples area for first values of $k<n$ computed with~\cite{sage}.}
 \label{table:kNaplesarea}
\end{table}

\section*{Acknowledgements}

Part of this research was performed with support from the Institute for Pure and Applied Mathematics (IPAM), which is supported by the National Science Foundation (Grant No.\ DMS-1440415), from the EDGE Foundation, and from private donations of Joan Barksdale and Nancy Sinclair.
ARVM was partially supported by NSF Graduate Research Fellowship DGE-1247392 and the NSF KY-WV LSAMP Bridge to Doctorate.  
The authors want to thank Raman Sanyal for sharing his work with Elias Drohla in private communication, and for pointing out the reference to \textit{Exercise 5.49(d,e)} in~\cite{EC2}, and Ayo Adeniran for his helpful insights throughout the project. ES would also like to thank Ralph Morrison for connecting her to PEH and
Bernd Sturmfels for funding through his BEAR fund at the University of California, Berkeley.

\bibliographystyle{alpha}  
\bibliography{team5biblio}

\newcommand{\etalchar}[1]{$^{#1}$}
\begin{thebibliography}{ABDB{\etalchar{+}}20}

\bibitem[ABDB{\etalchar{+}}20]{Adeniran1+}
A.~Adeniran, S.~Butler, G.~Dorpalen-Barry, P.~E. Harris, C.~Hettle, Q.~Liang,
  J.~L. Martin, and H.~Nam.
\newblock Enumerating parking completions using join and split.
\newblock {\em The Electronic Journal of Combinatorics}, 27(2):\#P2.44, 2020.

\bibitem[Ade20]{Adeniran4+}
A.~Adeniran.
\newblock {G}on\v{c}arov polynomials, partition lattices and parking sequences.
\newblock {\em Doctoral dissertation, Texas A\&M University}, 2020.
\newblock Available electronically from http://hdl.handle.net.

\bibitem[AY20a]{Adeniran2+}
A.~Adeniran and C.~Yan.
\newblock Gon\v{c}arov polynomials in partition lattices and exponential
  families.
\newblock {\em Advances in Applied Mathematics}, 2020.

\bibitem[AY20b]{Adeniran3+}
A.~Adeniran and C.~Yan.
\newblock On increasing and invariant parking sequences.
\newblock 2020.
\newblock https://arxiv.org/pdf/2005.04759.pdf.

\bibitem[Bau19]{BG}
A.~Baumgardner.
\newblock The {N}aples parking function.
\newblock {\em Honors Contract-Graph Theory, Florida Gulf Coast University},
  2019.

\bibitem[BCT10]{graphs}
B.~Benson, D.~Chakrabarty, and P.~Tetali.
\newblock {$G$}-parking functions, acyclic orientations and spanning trees.
\newblock {\em Discrete Math.}, 310(8):1340--1353, 2010.

\bibitem[Bon15]{bonaM}
M.~Bona.
\newblock {\em {Handbook of enumerative combinatorics}}.
\newblock Discrete Mathematics and Its Applications. CRC Press, Hoboken, NJ,
  2015.

\bibitem[CCH{\etalchar{+}}20]{PFCYOA}
J.~Carlson, A.~Christensen, P.E. Harris, Z.~Jones, and A.~Ramos Rodr\'{i}guez.
\newblock Parking functions: Choose your own adventure, 2020.

\bibitem[CDMY20]{Colaric2020IntervalPF}
E.~Colaric, R.~DeMuse, J.~Martin, and M.~Yin.
\newblock Interval parking functions.
\newblock {\em ArXiv}, 2020.
\newblock https://arxiv.org/pdf/2006.09321.pdf.

\bibitem[CHJ{\etalchar{+}}20]{NaplesPF}
A.~Christensen, P.~E. Harris, Z.~Jones, M.~Loving, A.~Ramos Rodr\'{i}guez,
  J.~Rennie, and G.~Rojas Kirby.
\newblock A generalization of parking functions allowing backward movement.
\newblock {\em The Electronic Journal of Combinatorics}, 27(1):\#P1.33, 2020.

\bibitem[KW66]{KW}
A.~Konheim and B.~Weiss.
\newblock An occupancy discipline and applications.
\newblock {\em SIAM Journal on Applied Mathematics}, 14(6):1266--1274, 1966.

\bibitem[Loe05]{dh}
N.~Loehr.
\newblock Combinatorics of {$q$}, {$t$}-parking functions.
\newblock {\em Adv. in Appl. Math.}, 34(2):408--425, 2005.

\bibitem[Pyk59]{Pyke}
R.~Pyke.
\newblock The supremum and infimum of the {P}oisson process.
\newblock {\em Ann. Math. Statist.}, 30:568--576, 1959.

\bibitem[RS{\etalchar{+}}]{FindStat}
M.~Rubey, C.~Stump, et~al.
\newblock {FindStat} - {T}he combinatorial statistics database.
\newblock Accessed July 2020.

\bibitem[Sch68]{rforests}
M.~Sch\"{u}tzenberger.
\newblock On an enumeration problem.
\newblock {\em J. Combinatorial Theory}, 4:219--221, 1968.

\bibitem[SD20]{SanyalDrohla}
R.~Sanyal and E.~Drohla.
\newblock Parking-functions, schattenvektoren und {C}atalan-zahlen.
\newblock {\em Bachelor thesis}, 2020.

\bibitem[Sea20]{sage}
W.~Stein~et al.
\newblock {\em {S}age {M}athematics {S}oftware ({V}ersion 9.0)}.
\newblock The Sage Development Team, 2020.

\bibitem[SP02]{PitmanStanley}
Richard~P. Stanley and Jim Pitman.
\newblock A polytope related to empirical distributions, plane trees, parking
  functions, and the associahedron.
\newblock {\em Discrete Comput. Geom.}, 27(4):603--634, 2002.

\bibitem[Sta96]{hyperplanes}
R.~Stanley.
\newblock Hyperplane arrangements, interval orders, and trees.
\newblock {\em Proc. Nat. Acad. Sci. U.S.A.}, 93(6):2620--2625, 1996.

\bibitem[Sta97]{ncp}
R.~Stanley.
\newblock {\em Parking functions and noncrossing partitions}, volume 4(2).
\newblock The Wilf Festschrift (Philadelphia, PA, 1996), 1997.

\bibitem[Sta99]{EC2}
R.~Stanley.
\newblock {\em Enumerative combinatorics. {V}ol. 2}, volume~62 of {\em
  Cambridge Studies in Advanced Mathematics}.
\newblock Cambridge University Press, Cambridge, 1999.

\end{thebibliography}

\end{document}